\documentclass[11pt,english]{amsart}
\usepackage{amsfonts, amssymb, amsmath, amsthm, latexsym, enumerate, epsfig, color, mathrsfs, fancyhdr, pdfsync, eurosym, endnotes, stmaryrd, bm}
\usepackage[all]{xy}
\usepackage[english]{babel}
\usepackage[mac]{inputenc}
\usepackage[T1]{fontenc}
\usepackage[percent]{overpic}
\usepackage{bbm}

\usepackage{hyperref}
\allowdisplaybreaks

\theoremstyle{plain}
\newtheorem{theorem}{Theorem}[section]
\newtheorem{lemma}[theorem]{Lemma}

\newtheorem{corollary}[theorem]{Corollary}
\newtheorem{definition}[theorem]{Definition}

\newtheorem{defn}[theorem]{Definition}
\newtheorem{thm}[theorem]{Theorem}

\theoremstyle{remark}
\newtheorem{remark}[theorem]{Remark}
\newtheorem{rmk}[theorem]{Remark}

\newcommand\Rc{\mathcal{R}}

\renewcommand\le{\leqslant}
\renewcommand\ge{\geqslant}

\def\P{{\mathbb{P}}}

\def\N{{\mathbb{N}}}

\def\sH{{\mathscr{H}}}

\def\sO{{\mathscr{O}}}

\def\cA{{\mathcal A}}\def\cE{{\mathcal E}}\def\cG{{\mathcal G}}\def\cH{{\mathcal H}}\def\cR{{\mathcal R}}\def\cS{{\mathcal S}}\def\cT{{\mathcal T}}\def\cV{{\mathcal V}}

\def\ie{\emph{i.e.}\;}\def\lc{\emph{loc.cit.}\;}\def\oc{\emph{op.cit.}\;}

\def\sh{{\textrm{Sh}}}

\def\vphi{\varphi}

\def\char{{\text{char}}}

\begin{document}
\setlength{\baselineskip}{0.55cm}
\title[Ramification locus of morphisms of Berkovich curves]{On the number of connected components of the ramification locus of a morphism of Berkovich curves}

\author{Velibor Bojkovi\'c }
\email{bojkovic@math.unipd.it}
\address{Dipartimento di matematica Tullio Levi-Civita, Universita di Padova, Via Trieste 63, 35121 Padova}

\author{J\'er\^ome Poineau}
\email{jerome.poineau@unicaen.fr}
\address{Laboratoire de math\'ematiques Nicolas Oresme, Universit\'e de Caen, BP 5186, F-14032 Caen Cedex}

\thanks{The second author was partially supported by the ANR project ``GLOBES'': ANR-12-JS01-0007-01 and ERC Starting Grant ``TOSSIBERG'': 637027.}


\begin{abstract}
Let~$k$ be a complete nontrivially valued non-archimedean field. Given a finite morphism of quasi-smooth $k$-analytic curves that admit finite triangulations, we provide upper bounds for the number of connected components of the ramification locus in terms of topological invariants of the source curve such as its topological genus, the number of points in the boundary and the number of open ends. 
\end{abstract}

\maketitle

\tableofcontents

\section{Introduction}

The study of the ramification locus of finite morphisms of Berkovich curves has been quite fruitful in recent years. Its properties have been used in a crucial way to prove non-archimedean analogues of classical results, such as the Riemann-Hurwitz formula (see \cite{CTT14,BojRH}), as well as to develop new directions of research, such as non-archimedean dynamics (see~\cite{BR} for a complete account and bibliographic references).

One of the things that make it such an inspirational area is the richness of the ramification locus of a finite morphism. Namely, it is a classical fact that given a finite morphism of Riemann surfaces, the set of ramified points (which both classically and in the setting of Berkovich curves can be defined as the points where the morphism locally fails to be an isomorphism) is discrete. In particular, if the surfaces are compact then this set of points is finite and subject to the bound coming from Riemann-Hurwitz formula. Contrary to this situation, the set of ramified points of a finite morphism of quasi-smooth $k$-analytic Berkovich curves (which are the non-archimedean analogues of Riemann surfaces) is very often infinite and not discrete.
  
 In order to give more detail, for the rest of the introduction, let us fix a complete nontrivially valued non-archimedean field~$k$ and a finite morphism of quasi-smooth $k$-analytic Berkovich curves $f:X\to Y$. Even if the ramification locus of~$f$ looks more complicated than its complex counterpart, it still enjoys nice topological and metric properties. For example, it is a closed subset of $X$ and it has a fairly nice ``radial'' structure. Loosely speaking, the latter means that there exists a skeleton of $X$ (a locally finite graph capturing the topological information of~$X$) around which the ramification locus is ``symmetric'' (this property being expressed in terms of a metric coming from the skeleton). For this aspect one may refer to Michael Temkin's paper \cite{temkinuniformization}.
 
 In the non-archimedean setting, much of the study of the ramification locus of a morphism has been initiated by Xander Faber in \cite{XF,XF2}, where the case of a rational map $f : \P^1_{k} \to \P^1_{k}$ is investigated. In \cite{XF}, among other things, the author is mainly occupied with the following question: what is the upper bound for the number of connected components of the ramification locus of such a rational function? He provides the answer $\deg(f)-1$ (see \cite[Theorem A]{XF}) and proves that it is optimal. We note here that the tools used in \oc  are quite technical relying on the existence of global coordinates on the projective line, hence hardly applicable to finite morphisms of higher genus curves. Although announced in \oc the study of this more general setting has not been carried out, and to the best of our knowledge little is known up to now. 

\smallskip

This article takes up the question of upper bounds for number of connected components of the ramification locus to its full generality. Namely, we consider $f:X\to Y$ to be a finite morphism of finite $k$-analytic Berkovich curves (\ie curves admitting finite triangulations, see Definition \ref{def:triangulation}). Our attention is restricted to this particular class of curves as going to a higher generality may lead to infinite number of components of the ramification locus, as a simple example in Section \ref{sec:nonfinite} shows. The main result, Theorem \ref{thm:ramification finite curves} provides the upper bound that we are looking for in terms of topological invariants of the source curve $X$ such as its topological genus, the number of points in its boundary and the number of ``open ends'', and furthermore the number of the ramified points in the latter two sets (where the morphism $f$ enters the picture). 

In a series of corollaries we specify our results for several common classes of $k$-analytic Berkovich curves, namely the affinoid, wide open and projective curves. In particular, our Theorem \ref{cor:generalization} reproves and generalizes Faber's result mentioned above: the ramification locus of a finite morphism $f:X\to Y$ between connected smooth projective $k$-analytic curves has at most $\deg(f) + g_{t}(X) -1$ connected components, where~$g_{t}(X)$ denotes the topological genus of the curve~$X$ (which is bounded by its geometric genus).

Actually, Theorem \ref{cor:generalization} gives more information and may lead to refined bounds in some particular cases. It implies, for instance, that, in the previous setting, if there exists a totally ramified point, then the bound drops to $g_{t}(X)+1$. Again, this generalizes a result of Faber in the case of finite endomorphisms of $\P^1_{k}$ (see \cite[Theorem C]{XF}).

\smallskip

Here is a brief description of the sections of the article and the tools used. The second section contains a recap on finite topological graphs and finite curves. In fact, we base our study mainly on two things. First, the graph-like structure of Berkovich curves, which allows to track the behavior of the topological invariants (topological genus, number of points in the boundary and number of open ends) in partitions of such curves with triangulations. Here, the incarnations of the localization exact sequence for Borel-Moore homology, namely Lemma \ref{lem:BM} for finite topological graphs and its equivalent Lemma \ref{lem:BM for curves} for finite curves, play the main role. Second, the existence of compatible skeleta (equivalently, as some readers may prefer, the existence of a semistable reduction of the morphism), which allows for simultaneous partition of both the source and target curve into simpler pieces. 

The main results, Theorem \ref{thm:ramification finite curves} and its corollaries are in the third section. The two above-mentioned properties allow for an induction to be carried out and we reduce the calculations to the simplest cases that are finite morphisms of open discs and annuli, where one easily obtains the desired results.

\section{Topological graphs and finite $k$-analytic curves}

\subsection{Finite topological graphs}

\subsubsection{Topological realizations of finite graphs}

\begin{definition}
A {\em finite topological graph} is a Hausdorff topological space~$\cG$ with a distinguished finite subset~$V(\cG)$ whose complement has finitely many connected components, each of them homeomorphic to the real interval $(0,1)$.

The elements of~$V(\cG)$ are called the {\em vertices} of~$\cG$ and its cardinality is denoted by~$v(\cG)$. The connected components of $\cG\setminus V(\cG)$ are called the \em{edges} of~$\cG$, the set they define is denoted by~$E(\cG)$ and its cardinality by~$e(\cG)$.
\end{definition}

The set of {\em open ends} of the finite graph~$\cG$ is 
\[End(\cG) := \varprojlim_{K} \pi_{0}(\cG \setminus K),\]
where~$K$ runs through the compact subsets of~$\cG$. We denote its cardinality by~$e^0(\cG)$. It obviously only depends on the topological space~$\cG$.

We denote by~$E^0(\cG)$ (resp. $E^{00}(\cG)$) the subset of~$E(\cG)$ consisting of the edges that are not relatively compact in~$\cG$ (resp. edges that are connected components of~$\cG$). Note that~$E^{00}(\cG)$ is a subset of~$E^0(\cG)$. The number of open ends of~$\cG$ may be computed in a concrete way by 
\[e^0(\cG) = \#E^0(\cG) +  \#E^{00}(\cG).\]
In particular, note that $e^0(\cG)$ is equal to the cardinality of $E^0(\cG)$ when~$\cG$ has no edge that is a connected component, but only in this case.

It is clear that the same topological space $\cG$ may be considered as a finite topological graph for different sets of vertices. We will often take advantage of this situation and consider $\cG$ as a finite topological graph with respect to some (unspecified) vertex set.

Any connected component~$\cG_{0}$ of~$\cG$ is naturally a finite topological graph as well, with $V(\cG_{0}) = V(\cG) \cap \cG_{0}$ and $E(\cG_{0}) = \{e \in E(\cG) \mid e \subset \cG_{0}\}$. Similarly, for any finite subset~$V'$ of~$\cG$, the topological space $\cH := \cG\setminus V'$ is naturally a finite topological graph, where $V(\cH) = V(\cG) \setminus V'$ and $E(\cH)$ is the set of connected components of~$\cH\setminus V(\cH)$.

For each positive integer~$m$, set 
\[U_{m} := \{r \exp(2i q\pi/m) \mid r\in[0,1), q\in\{0,\dotsc,m-1\}\}.\] For each $x\in\cG$, there exists a unique positive integer~$v(x)$ such that there exists a homeomorphism from a neighborhood~$U$ of~$x$ in~$\cG$ to~$U_{v(x)}$ sending~$x$ to~0. We call~$v(x)$ the valence of~$x$. There are only finitely many points in~$\cG$ whose valence is different from~2, as this is a subset of the vertex set. Points with valence $1$ are called endpoints.

We define a compactification~$\overline{\cG}$ of~$\cG$ as the open end compactification of~$\cG$. It is obtained from~$\cG$ by adding exactly one point for each open end of~$\cG$. Note that we have natural bijections $\pi_{0}(\cG) = \pi_{0}(\overline{\cG})$ and $\pi_{1}(\cG) = \pi_{1}(\overline{\cG})$ .

The topological space~$\overline{\cG}$ naturally carries a structure of finite topological graph where $V(\overline{\cG})$ is naturally in bijection with $V(\cG) \sqcup End(\cG)$ and $E(\overline{\cG})$ is naturally in bijection with $E(\cG)$.

\subsubsection{Genus of a finite topological graph} 

Let $\cG$ be a finite topological graph. We define its {\em genus}  to be the first Betti number of the underlying topological space $\cG$ and denote it by~$g(\cG)$. 

Recall that, when the graph~$\cG$ is compact, it may be computed by the classical formula
\[g(\cG) = \#\pi_{0}(\cG) - v(\cG) + e(\cG).\]
A similar formula may be derived for an arbitrary finite topological graph.

\begin{lemma}\label{lem:genus}
We have
\[g(\cG) = \#\pi_{0}(\cG) - v(\cG) + e(\cG) - e^0(\cG).\]
\end{lemma}
\begin{proof}
The result follows from the formula for the compact case applied to the compactification~$\overline\cG$ of~$\cG$ and the relations
\[\begin{cases}
\pi_{0}(\overline\cG) = \pi_{0}(\cG);\\
g(\overline\cG) = g(\cG);\\
v(\overline\cG) = v(\cG) + e^0(\cG);\\
e(\overline\cG) = e(\cG).
\end{cases}\]
\end{proof}

\subsubsection{Genus and partitions of finite topological graphs}

\begin{lemma}\label{lem:BM}
Let~$\cG$ be a finite topological graph and let~$S$ be a subset of $V$. Set $\cH := \cG \setminus S$ and denote by $\cH_{1},\dotsc,\cH_{t}$ its connected components. Then, we have
\begin{align*}
\#\pi_{0}(\cH) - g(\cH) - e^0(\cH)  &= \sum_{i=1}^t (1-g(\cH_{i}) - e^0(\cH_{i}))\\ 
&= \#\pi_{0}(\cG) - g(\cG) - e^0(\cG) - \# S.
\end{align*}
\end{lemma}
\begin{proof}
The first line of the formula being straightforward, we only prove the second. Endowing~$\cH$  with the structure of topological graph induced by that of~$\cG$, we have 
\[\begin{cases}
e(\cH) = e(\cG);\\
v(\cH) = v(\cG) - \#S.
\end{cases}\]
By Lemma~\ref{lem:genus}, we then have 
\begin{align*}
\#\pi_{0}(\cH) - g(\cH) - e^0(\cH)&= v(\cH) - e(\cH) \\
&= v(\cG) - e(\cG) - \#S\\
&= \#\pi_{0}(\cG) - g(\cG) - e^0(\cG)- \#S.
\end{align*}
\end{proof}

\begin{remark}
For a finite topological graph~$\cG$, the quantity $\pi_{0}(\cG) - g(\cG) - e^0(\cG)$ is nothing but the Euler characteristic of~$\cG$ for the Borel-Moore homology. Noting this, the formula of Lemma~\ref{lem:BM} follows from a localization exact sequence.
\end{remark}

\subsection{(Finite) $k$-analytic curves}

Throughout the paper, $k$ is an algebraically closed, complete, non-archimedean and non-trivially valued field of arbitrary characteristic. We assume some familiarity with $k$-analytic spaces in the sense of Berkovich, in particular the structure results of quasi-smooth $k$-analytic curves and classifiction of their points as presented, for example, in \cite[Chapter 4]{Ber90} and \cite[Section 3.6]{BerCoh}. A comprenhensive study of $k$-analytic curves is presented in the upcoming book \cite{Duc-book} to which we will often refer. 

One of the indispensable tools to understand the structure of quasi-smooth $k$-analytic curves is the notion of triangulation. We warn the reader here that for the sake of simplification of the statements, we will consider punctured open discs to be open annuli, which differs from the standard terminology. 

\begin{definition}\label{def:triangulation}
A {\em triangulation} of a quasi-smooth $k$-analytic curve~$X$ is a locally finite set $\cT$ of points of type 2 and 3 in~$X$ such that $X\setminus\cT$ is a disjoint union of open discs and open annuli. 
\end{definition}

Every quasi-smooth $k$-analytic curve admits a triangulation \cite[Th\'eor\`eme~5.1.14]{Duc-book}. However, in light of Section~\ref{sec:nonfinite}, to study the number of components of ramification locus, we will focus on the following class of quasi-smooth $k$-analytic curves: finite curves.

\begin{definition}
 We say that a quasi-smooth $k$-analytic curve $X$ is {\em finite}, or {\em finitely triangulated} if $X$ admits a finite triangulation.  
\end{definition}

\begin{remark}\label{rem:finitelymanyannuli}
Let~$X$ be a quasi-smooth $k$-analytic curve and~$\cT$ be a triangulation of~$X$. For each $t\in \cT$, among the connected components of~$X\setminus\cT$ whose closures contain~$t$, there can only be finitely many open annuli. This follows for instance from the description of bases of neighborhoods of points (see~\cite[Th\'eor\`eme~4.5.4]{Duc-book}).

In particular, if~$X$ is finite, then there are only finitely many connected components of $X\setminus\cT$ that are open annuli.
\end{remark}

We have the following characterization of finite curves. The case where~$X$ is strictly $k$-analytic and~$k$ has characteristic~0 is \cite[Theorem 1.4.2]{BojRH}. Our proof follows the same strategy.

\begin{theorem}\label{thm:finite curves}
 A connected quasi-smooth $k$-analytic curve $X$ is finite if and only if it is isomorphic to the complement of finitely many type $1$ points, closed and open discs in some smooth projective $k$-analytic curve $X'$.
\end{theorem}
\begin{proof}
Suppose that $X$ is a complement in a smooth projective $k$-analytic curve $X'$ of union of elements in sets $P$, $O$ and $C$, where $P$ is a finite set of type 1 points, $O$ is a finite set of open discs in $X'$ and $C$ is a finite set of closed discs in $X'$. We may assume that the elements of the sets $P$, $O$ and $C$ are disjoint. For each $p\in P$, let $D_p$ be an open disc in $X'$ that contains $p$. Up to shrinking those discs, we may assume that they are disjoint and disjoint from all the discs in $O$ and $C$. 
Let $S_1$ be the set of endpoints in $X'$ of the open discs $D_p$, for $p\in P$, let $S_2$ be the set of endpoints in $X'$ of open discs in $O$ and further let $S_3$ be the set of Shilov points of the closed discs in $C$. Finally, let $S:=S_1\cup S_2\cup S_3$. Let $\cS'$ be a triangulation of $X'$ that contains $S$ and has an empty intersection with the discs $D_p$, for $p\in P$. We can always achieve this by  starting with any triangulation $\cS''$ of $X'$ that contains $S$ and deleting the points that belong to $D_p$. 

We claim that $\cS:=\cS'\cap X$ is a finite triangulation of $X$. Indeed, the set $\cS'\cap X$ is a finite set of type 2 and 3 points since it is a subset of $\cS'$. Let $U$ be a connected component of $X\setminus\cS$. By the construction of $X$ and $\cS$, if $U$ is disjoint from the discs $D_p$, for $p\in P$, then it is a connected component of $X'\setminus \cS'$, hence an open annulus or an open disc. On the other hand, if there exists~$p\in P$ such that~$U$ meets~$D_{p}$, then we have $U=D_p\setminus\{p\}$.

\smallbreak

For the other direction, let~$X$ be a finite, connected $k$-analytic curve. If ~$X$ is a point, then it is necessarily of type 3 and the result holds in this case. So, assume that $X$ is not a point and let~$\cT$ be a finite triangulation of~$X$. Let~$C$ be a connected component of~$X\setminus \cT$ that is not relatively compact in~$X$. If it is isomorphic to an open disc, then we have $X=C$ and the result holds. 

We may now assume that each connected component of $X\setminus \cT$ that is not relatively compact in~$X$ is isomorphic to an open annulus. Each of these connected components may be compactified by adding a point of type~1 or a closed disc. By Remark~\ref{rem:finitelymanyannuli}, there are only finitely many such components, hence the curve~$X$ itself may be compactified to a curve~$X^0$ by adding finitely many points of type~1 and closed discs.  By \cite[Corollaire~6.1.4]{Duc-book}, $X^0$ is either projective or affinoid. In the former case, we are done, so let us assume that~$X^0$ is affinoid. 

 Let~$x$ be a type~3 point in the Shilov boundary~$\sh(X^0)$ of~$X^0$. By assumption, it is not isolated, hence, by \cite[Th\'eor\`eme~4.3.5]{Duc-book}, it has a neighborhood isomorphic to a closed annulus~$A$ whose Shilov boundary contains~$x$. By glueing an open disc to~$X^0$ along the boundary of~$A$, we obtain a curve~$X^{0x}$ which is either projective or affinoid with Shilov boundary $\sh(X^{0x}) = \sh(X^0) \setminus \{x\}$. By applying this process repeatedly, we get a curve~$X^1$, obtained from~$X$ by adding finitely many points of type~1 and closed or open discs, which is either a projective curve or an affinoid curve with no type~3 points in its Shilov boundary. In the former case, we are done, so let us assume that~$X^1$ is affinoid. 

Since~$X^1$ has only type~2 points in its Shilov boundary, it is strictly $k$-affinoid. A result of Van der Put (see \cite[Theorem 2.1]{VdP90}) then ensures that we can get a projective curve from~$X^1$ by adding finitely many open discs. This finishes the proof.

\end{proof}

In particular, it follows that the following classes of curves are finite:
\begin{itemize}
\item Compact quasi-smooth $k$-analytic curves. The fact that they admit triangulations together with compactness implies that they admit finite triangulations. If such a curve is connected, then it is either affinoid or projective by \cite[Corollaire~6.1.4]{Duc-book}. 
\item Complement of a finite nonempty family of closed discs and  rational points in a smooth projective $k$-analytic curve. We will call these curves {\em wide open curves}, although a wide open curve is classically a curve which is a complement of finitely many strict closed discs in a smooth projective $k$-analytic curve. Note that the class of wide open curves does not contain projective $k$-analytic curves, but analytifications of smooth affine $k$-algebraic curves are in this class.
\end{itemize}

\subsubsection{Boundary annuli} 

If $X$ is a finite curve, we will commonly refer to the annuli which are not relatively compact in $X$ and which are not connected components of~$X$ as {\em boundary annuli}. 
 
 We divide boundary annuli in classes by saying that two boundary annuli $A_1$ and $A_2$ belong to the same class if $A_1\subseteq A_2$ or $A_2\subseteq A_1$. We denote the set of classes of boundary annuli by $\cE^0(X)$ and its cardinality by~$e^0(X)$. For $v\in \cE^0(X)$, we will write $A_v$ for a boundary annulus belonging to the class $v$.

\subsubsection{Boundary of a curve} 

Berkovich defined the boundary~$\partial(X)$ of a $k$-analytic space~$X$ in~\cite[Section 2.5]{Ber90} and~\cite[Definition 1.5.4]{BerCoh}. Recall that, when~$X$ is an affinoid curve, $\partial(X)$ coincides with the Shilov boundary~$\sh(X)$ (see \cite[Lemma~2.3]{RodriguezNormal} and~\cite[Lemma~2.4]{MehmetiPatching}). It then follows from \cite[Proposition~1.5.5]{BerCoh} that, when $X$ is a $k$-analytic curve, $\partial(X)$ is the set of points~$x$ in~$X$ that have an affinoid neighborhood whose Shilov boundary contains~$x$.

When~$X$ is a quasi-smooth $k$-analytic curve, any triangulation of~$X$ contains the boundary. In particular, if~$X$ is finite, then~$\partial(X)$ is finite too.

%
%
%

\subsubsection{Topological genus of finite curves}

The {\em skeleton} of an open (resp. closed) annulus~$A$ is the subset of points that have no neighborhoods isomorphic to a disc. It is an open (resp. closed) interval. We denote it by~$\Gamma_{A}$.

Similarly, given a quasi-smooth $k$-analytic curve $X$ and a triangulation $\cT$ of $X$, we may form the {\em skeleton of $X$ with respect to $\cT$} (or induced by $\cT$), denoted by $\Gamma_\cT(X)$, by taking the complement of all the open discs in $X$ not intersecting~$\cT$. In different words, $\Gamma_{\cT}(X)$ consists of the triangulation $\cT$ and the skeleta of all the open annuli which are connected components of $X\setminus\cT$. 

In general, any subset $\Gamma$ of $X$ such that $\Gamma=\Gamma_\cT(X)$, where $\cT$ is a triangulation of $X$, will be called a skeleton of $X$. 

Assume that~$\Gamma_{\cT}(X)$ meets every connected component of~$X$. For each $x \in X\setminus \Gamma_{\cT}(X)$, the connected component of $X\setminus \Gamma_{\cT}(X)$ containing~$x$ is an open disc whose boundary point belongs to~$\Gamma_{\cT}(X)$. We denote the latter by~$r_{\cT}(x)$. For each $x \in \Gamma_{\cT}(X)$, we set $r_{\cT}(x) = x$. Then, the map $r_{\cT}: x\in X \mapsto r_{\cT}(x) \in \Gamma_{\cT}(X)$ is a deformation retraction. 

Assume that~$X$ is a finite curve. Then, the corresponding skeleton $\Gamma_{\cT}(X)$ has a natural structure of finite topological graph, with $V(\Gamma_{\cT}(X)) = \cT$. Note that $E(\Gamma_{\cT}(X))$ is the set of skeleta of connected components of $X\setminus\cT$ which are open annuli.

\begin{defn}
 The {\em topological genus} of $X$, denoted by $g_t(X)$, is the first Betti number of the underlying topological space of $X$.
\end{defn}

Note that we have 
\[g_{t}(X) = g(\Gamma_{\cT}(X)).\] 
To prove this, one may assume that~$X$ is connected. If~$\Gamma_{\cT}(X)$ is nonempty, the formula follows from the existence of the retraction~$r_{\cT}$. If~$\Gamma_{\cT}(X)$ is empty, then~$X$ is an open disc and the result is obvious. 

By Lemma~\ref{lem:genus}, we deduce that
\[g_{t}(X) = \#\pi_{0}(X) - \# \cT + e(\Gamma_{\cT}(X)) - e^0(\Gamma_{\cT}(X)).\]
This expression will be used in several places in the article.

We also note that the map sending a boundary annulus to its skeleton induces a bijection between the set of classes of boundary annuli of~$X$ and the set of open ends of~$\Gamma_{\cT}(X)$. In particular, we have $e^0(\Gamma_{\cT}(X)) = e^0(X)$. 

The following lemma corresponds to Lemma \ref{lem:BM}.
\begin{lemma}\label{lem:BM for curves}
Let $W$ be a finite curve and let $S$ be a finite set of type 2 and 3 points in $W$. Then, the open subset $V:=W\setminus S$ has only finitely many connected components which are not open discs. Denoting them by $V_1,\dots,V_m$, we have
\[\sum_{i=1}^m (1-g_t(V_{i}) - e^0(V_{i}))
= \pi_{0}(W) - g_t(W) - e^0(W) - \# S.\]
\end{lemma}
\begin{proof} Let us first prove that there are indeed only finitely many components of $V$ which are not open discs. To this end, let $\cS$ be any finite triangulation of $W$ which contains~$S$. 

The points of $\cS\setminus S$ being contained in finitely many connected components of $V$, there are only finitely many connected components of $V$ which are not connected components of $V \setminus \cS$. Since there are only finitely many connected components of $V \setminus \cS$ which are not open discs, the same result holds for $V$.

We continue the proof of lemma. Let $\cS$ be a sufficiently big finite triangulation of $W$ such that $\cS$ contains $S$ and $\cS\cap V_i$ is non-empty for each $i=1,\dots,t$. Denote by~$\cV$ the set of connected components of~$W\setminus S$ that meet~$\cS$. Note that, for each $V\in\cV$, $\cS\cap V$ is a triangulation of $V$. If we denote by $\Gamma_\cS$ the skeleton of $W$ with respect to $\cS$ and by $\Gamma_{\cS\cap V}$ the skeleton of $V$ with respect to $\cS\cap V$, then the connected components of $\Gamma_{\cS}\setminus S$ are the $\Gamma_{\cS\cap V}$, for $V \in \cV$. It now follows from Lemma \ref{lem:BM} that 
\[\sum_{V\in\cV} (1-g_t(V) - e^0(V))
= \pi_{0}(W) - g_t(W) - e^0(W) - \# S.\]

If $V \in \cV$ is not one of the~$V_{i}$'s, then $V$ is an open disc, hence $g_t(V)=0$ and $e^0(V)=1$. The result of the lemma follows.
\end{proof}

\subsection{Compatible skeleta and triangulations}

\begin{defn}\label{def:comp triang}
 Let $f:X\to Y$ be a finite morphism of quasi-smooth $k$-analytic curves. 
 We say that skeleta $\Gamma_1$ of $X$ and $\Gamma_2$ of $Y$ are {\em $f$-compatible} if $\Gamma_2=f(\Gamma_1)$ and $\Gamma_1=f^{-1}(\Gamma_2)$. Similarly, we say that triangulations $\cS$ of $X$ and $\cT$ of $Y$ are {\em $f$-compatible} (resp. {\em strictly $f$-compatible}) if $\cT=f(\cS)$ and $\cS=f^{-1}(\cT)$ (resp. $\cT=f(\cS)$ and $\cS=f^{-1}(\cT)$ and no two adjacent points in $\cS$ are mapped to the same point in $\cT$). 
 \end{defn}
 
 We note that if $\cS$ and $\cT$ are $f$-compatible triangulations of $X$ and $Y$, respectively, then the property that $\Gamma_{\cS}$ and $\Gamma_{\cT}$ are $f$-compatible is slightly finer than the strict $f$-compatibility of $\cS$ and $\cT$. For a trivial but illustrative example we may take $f:X\to Y$ to be a finite morphism of an open annulus to an open disc, and put $\cS=\cT=\emptyset$. 
  
 For a given finite morphism $f:X\to Y$ of quasi-smooth $k$-analytic curves there exist $f$-compatible skeleta as well as strictly $f$-compatible triangulations containing any given finite sets of type two and three points in~$X$ and~$Y$. This can be fairly easily deduced from the existence of triangulations of quasi-smooth $k$-analytic curves (see \cite[Corollary~4.26]{ABBR} or \cite[Section~3.5.11]{CTT14}). It is also related to the stable reduction of finite morphisms (for this aspect see \cite[Section 5]{ABBR} and references therein).
  
 Even more, if $Y$ and $X$ are finite curves, then we may choose $\cS$ and $\cT$ to be finite sets. These facts have many avatars in the literature depending whether we use formal models, triangulations or skeleta to capture the nice structure of quasi-smooth $k$-analytic curves. However, in this article, we are mainly interested in the following, certainly well-known, consequence. We agree that whenever we speak about compatible triangulations of finite curves, the triangulations will be assumed to be finite as well.
 
 \begin{corollary}\label{cor:compa partitions}
  Let $f:X\to Y$ be a finite morphism of finite curves. 
  \begin{itemize}
  \item[(a)] Let $\cS$ and $\cT$ be strictly $f$-compatible triangulations which induce $f$-compatible skeleta of $X$ and $Y$, respectively.   
  \begin{itemize}
   \item[(i)] Let $A$ be a connected component of $X\setminus \cS$ which is an open annulus. Then, $f(A)$ is a connected component in $Y\setminus\cT$ and it is an open annulus. If $A$ is not relatively compact in $X$, then neither is $f(A)$ in $Y$.
   \item[(ii)] Let $B$ be a connected component of $Y\setminus\cT$ which is an open annulus. Then $f^{-1}(B)$ is a disjoint union of connected components in $X\setminus\cS$, all of which are open annuli. If~$B$ is not relatively compact in $Y$, then neither is any of the components of $f^{-1}(B)$ in~$X$.
  \end{itemize}
  \item[(b)] Let $x\in X$ be a point and put $y:=f(x)$. Then, there exist open neighborhoods $U_x$ of $x$ in $X$ and $U_y$ of $y$ in $Y$ such that $f_{|U_x}:U_x\to U_y$ is a finite morphism and $y$ has only one preimage in $U_x$.
  \end{itemize}
 \end{corollary}
\begin{proof}
 (a) If $A$ is an open annulus which is a connected component of $X\setminus\cS$, then it is mapped to a connected component of $Y\setminus\cT$, so $f(A)$ is either an open annulus either an open disc. But, since $A$ has a nonempty skeleton and $\Gamma_{\cS}$ and $\Gamma_{\cT}$ are $f$-compatible, $f(A)$ cannot be an open disc. 
 
 If $A$ is relatively compact, then $f(A)$ has to be relatively compact as well because $f$ is continuous. If $f(A)$ is relatively compact, then~$A$ is relatively compact because~$f$ is finite, hence proper.
 
 This proves assertions $(i)$ and $(ii)$.
 
 (b) Let~$V$ be a compact neighborhood of~$y$ in~$Y$. Then $W := f^{-1}(V)$ is a compact neighborhood of $F := f^{-1}(y)$ in~$X$. Since~$F$ is finite and~$X$ is Hausdorff, for each $z\in F$, there exists an open subset~$O_{z}$ of~$X$ such that $O_{z} \cap F = \{z\}$. Set $K := W \setminus \bigsqcup_{z\in F} O_{z}$. It is a compact subset of~$X$ and its image~$L$ is a compact subset of~$Y$ that does not contain~$y$. The complement of~$L$ in~$W$ contains an open neighborhood~$U_{y}$ of~$y$ in~$Y$. By construction, the connected component~$U_{x}$ of $f^{-1}(U_{y})$ containing~$x$ is contained in~$O_{x}$, hence $U_{x} \cap F = \{x\}$, which finishes the proof.
 \end{proof}

\subsubsection{Morphisms of boundary annuli} 

It follows from Corollary~\ref{cor:compa partitions} that, for every $v\in \cE^0(X)$, there exists $A_v$ such that $f_{|A_v}:A_v\to f(A_v)$ is a finite morphism of boundary annuli. Moreover, for each boundary annulus $A\subseteq A_v$, $f_{|A}:A\to f(A)$ is again a finite morphism of annuli (see Lemma~\ref{lem:annuli}).

We consider all the boundary annuli $A$ in $X$ with this property: the restriction $f_{|A}:A\to f(A)$ is a finite morphism of open annuli. We say that two such annuli belong to the same class if and only if one of them is a subset of the other. The set of all classes is denoted by $\cE^0_f(X)$. We note that $\#\cE^0_f(X) = \#\cE^0(X) = e^0(X)$.

\subsection{Ramification locus}

\begin{definition}\label{def:multiplicity}
Given a finite morphism $f:X\to Y$ of quasi-smooth $k$-analytic curves, we define the {\em multiplicity} of a point $x\in X$ as the number
\[
\mu_f(x):=\begin{cases}
	  e, \quad \text{$x$ of type $1$ and} \quad m_{f(x)}=m_x^e\sO_{f(x)},\\
          [\sH(x):\sH(f(x))],\quad \text{otherwise}.
         \end{cases}
\]
We say that $x\in X$ is a {\em ramified point} if $\mu_f(x)>1$.
\end{definition}

Let $f:X\to Y$ be a finite morphism of quasi-smooth $k$-analytic curves. 
The following properties of the multiplicity are not difficult to prove (see \cite[Section 6.3]{BerCoh}).

For every point $y\in Y$, we have the equality
\[\deg(f)=\sum_{x\in f^{-1}(y)}\mu_f(x).\]

The multiplicity of a point $x\in X$ is the maximal integer $n$ such that for every neighborhood $U$ of $x$ there exists a point $y\in Y$ such that $y$ has $n$ inverse images in $U$. In fact, one may choose $y$ to be a rational point in $Y$.

Finally, by choosing $U_x$ and $U_y$ as in Corollary \ref{cor:compa partitions} (b), we see that $x$ is a ramified point if and only if $f$ is not an isomorphism in a neighborhood of $x$, which is more in the spirit of the classical notion of ramification and ramified point.

\begin{rmk}
 The multiplicity of a point in the context of finite (\'etale) morphisms of Berkovich curves was first defined by V. Berkovich in \cite[Section~6.3]{BerCoh}, where it was named geometric ramification index. Some of its basic properties, together with the ones we listed above are stated in Remarks 6.3.1. in \lc.
\end{rmk}

A closely related function is the following.
\begin{definition}
 For a point $x\in X$, the {\em ramification indicator function} $\delta_m$ is defined as 
 \[
 \delta_m(x):=\begin{cases}
              0 \quad\text{if } \mu_f(x)=1;\\
              1\quad\text{otherwise}.
             \end{cases}
 \]
 More generally, for a $k$-analytic subset $U$ of $X$, such that $f_{|U}:U\to f(U)$ is a finite morphism, we define
 \[
 \delta_m(U):=\begin{cases}
              0 \quad\text{if } \deg(f_{|U})=1;\\
              1\quad\text{otherwise}.
             \end{cases}
 \]
\end{definition}

Note that, for $v\in \cE^0_f(X)$ and any open annulus $A_{v}$ in the class $v$, the degree of the morphism $f_{|A_v}$ only depends on $v$ and not on $A_{v}$, so it makes sense also to use the notation $\delta_m(v)$ for $\delta_m(A_v)$.

Finally, we define the main object of study of this paper.

\begin{definition}
Given a finite morphism $f:X\to Y$ of quasi-smooth $k$-analytic curves, we define the {\em ramification locus of $f$} to be the set $\cR_f:=\{x\in X\mid \mu_f(x)>1\}$. 

If $U$ is a $k$-analytic subset of $X$ such that $f_{|U}:U\to f(U)$ is a finite morphism, we define $r_f(U)$ to be the cardinal of the set of connected components of the ramification locus $\cR_{f_{|U}}=U\cap \cR_f$. Also, we denote by $r'_f(U)$ the number of the connected components of $\cR_{f_{|U}}=U\cap \cR_f$ which are relatively compact in $U$. 
\end{definition}

\begin{remark}
Note that we have $r_{f}(U) \le r'_f(U) + \sum_{v\in \cE^0(U)}\delta_m(v)$.
\end{remark}

\section{Number of components of the ramification locus}

\subsection{Infinite number of components of the ramification locus}\label{sec:nonfinite} 

The next example shows that the number of components of the ramification locus may be infinite if we do not restrict ourselves to the class of finite curves. 

Our strategy is the following. We start with a curve and a generically \'etale cover with non-empty ramification locus. We can arrange that the original curve has (at least) two unramified open ends. We then take countably many copies of it and construct a new curve by glueing one unramified end of one copy to the other unramified end of the next one. The cover naturally extends to the new curve and its ramification locus has infinitely many connected components.

Assume for simplicity that the characteristic of the residue field $\tilde{k}$ is different from 2 (we invite the reader to construct a similar example if $\char(\tilde{k})=2$) and let $f':A'\to B'$ be a Kummer map of degree~2 between closed annuli. Due to our assumptions on the characteristic of the residue field, the ramification locus of the morphism $f'$ is the skeleton~$\Gamma$ of~$A'$, so that for every open disc $D$ in $A'$, the restriction $f'_{|D}$ is an isomorphism.

Let~$D^1$ and~$D^2$ be two distinct connected components of $B' \setminus \Gamma_{B'}$. They are isomorphic to the open unit disc. Let~$D^1_{c}$ and~$D^2_{c}$ be closed subdiscs of~$D^1$ and~$D^2$ respectively such that the annuli $C^1:=D^1\setminus D^1_{c}$ and $C^2:=D^2\setminus D^2_{c}$ are isomorphic. Note that~$f'$ is a local isomorphism above~$C^1$ and~$C^2$. 

Let $E^{1,+}$ and $E^{1,-}$ (resp. $E^{2,+}$ and~$E^{2,-}$) be the connected components of~$f'^{-1}(C^1)$ (resp. $f'^{-1}(C^2)$). They are annuli. For each $j \in\{1,2\}$ and each $\sigma\in\{-,+\}$, $f'$ induces an isomorphism $\varphi^{j,\sigma} : E^{j,\sigma} \xrightarrow[]{\sim} C^j$.

We set $B := B' \setminus(D^1_{c}\cup D^2_c)$ and $A:=f'^{-1}(B)$. We denote by $f : A\to B$ the finite morphism induced by~$f'$. We note that the ramification locus of~$f$ is still $\Gamma$. The reader may refer to Figure 1 for a visualization of the morphism $f$ (The ramification locus $\Gamma$ is represented with a thicker line).
\begin{center}
  \begin{figure}
  \begin{overpic}[width=0.45\textwidth,tics=10]{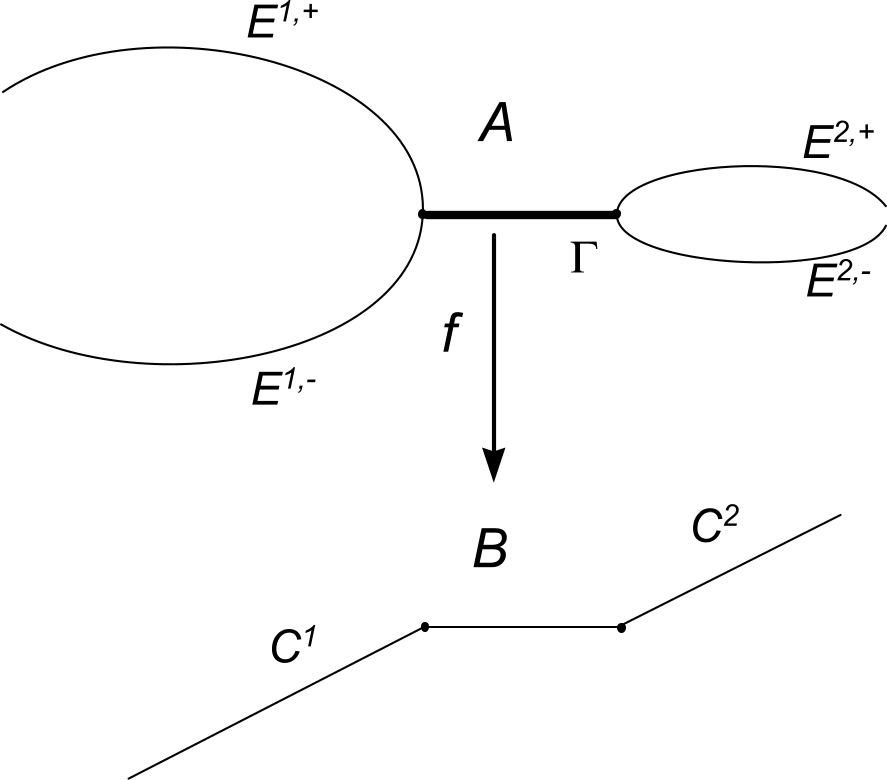}
  \end{overpic}
  \caption{Skeletal representation of the morphism $f:A\to B$.}
  \end{figure}
 \end{center}

For each $i\in \N$, we now consider a copy of the previous construction. We denote the corresponding objects with a subscript~$i$: $A_{i}$, $B_{i}$, $\Gamma_{i}$, $f_{i}$, etc. We also choose an isomorphism $\gamma_{i} : C_{i+1}^1 \xrightarrow[]{\sim} C^2_{i}$ such that the image of a sequence in~$D_{i+1}^1$ approaching the skeleton of~$B_{i+1}$ goes away from the skeleton of~$B_{i}$. 

We now construct a quasi-smooth $k$-analytic curve~$Y$ by identifying, for each $i\in \N$, the annulus~$C_{i+1}^1$ in~$A_{i+1}$ with the annulus~$C_{i}^2$ in~$A_{i}$ \textit{via} $\gamma_{i}$. Similarly, we construct a quasi-smooth $k$-analytic curve~$X$ by identifying, for each $i\in \N$, the annulus~$E_{i+1}^{1,+}$ in~$B_{i+1}$ with the annulus~$E_{i}^{2,+}$ in~$B_{i}$ \textit{via} $\beta_i^+:=(\varphi_{i}^{2,+})^{-1} \circ \gamma_{i} \circ \varphi_{i+1}^{1,+}$ and the annulus~$E_{i+1}^{1,-}$ in~$B_{i+1}$ with the annulus~$E_{i}^{2,-}$ in~$B_{i}$ \textit{via} $\beta_i^-:=(\varphi_{i}^{2,-})^{-1} \circ \gamma_{i} \circ \varphi_{i+1}^{1,-}$. The gluing process is represented in Figure 2.

\begin{center}
  \begin{figure}[h]
  \begin{overpic}[width=0.7\textwidth,tics=10]{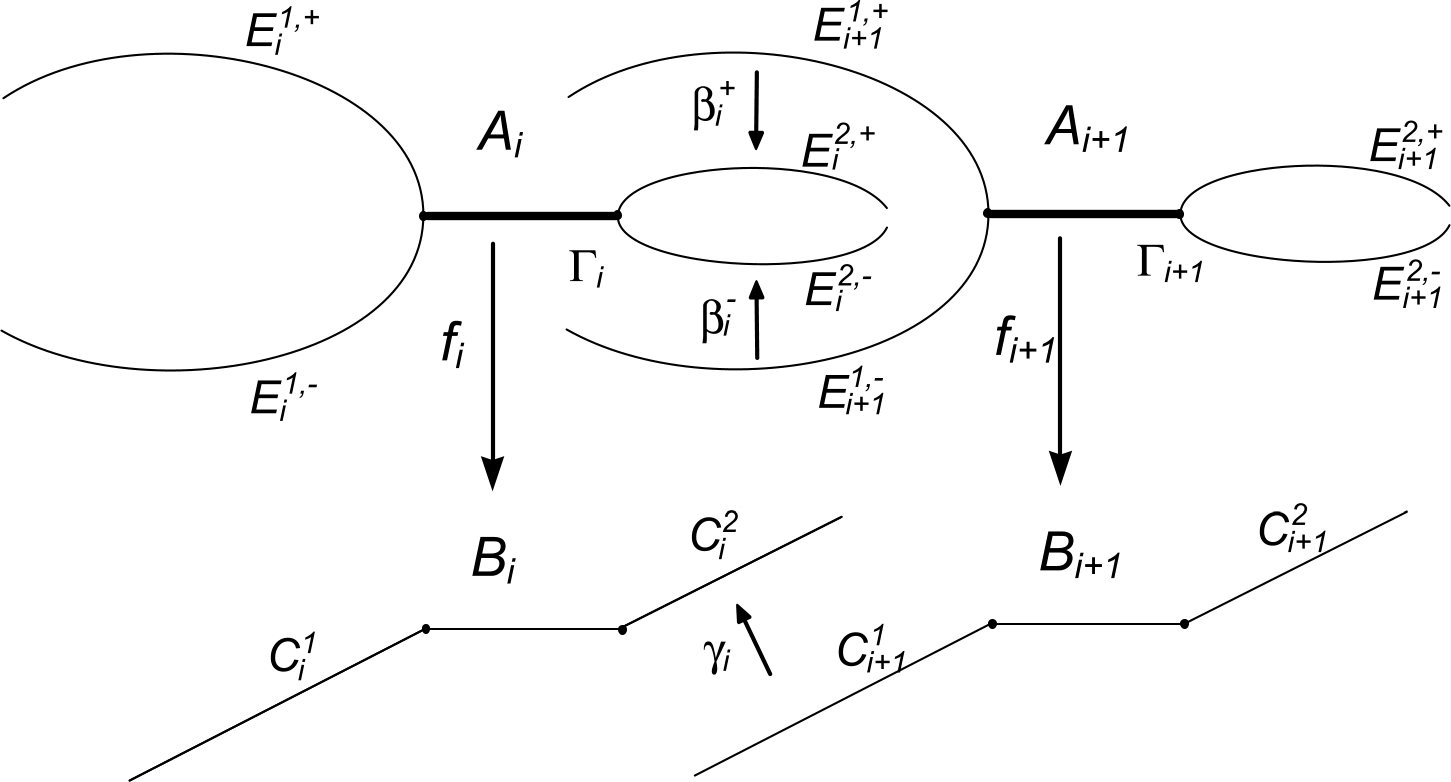}
  \end{overpic}
  \caption{Gluing of the morphisms $f_i:A_i\to B_i$ and $f_{i+1}:A_{i+1}\to B_{i+1}$.}
  \end{figure}
 \end{center}

By construction, the family of morphisms~$(f_{i})_{i\in\N}$ induces a finite morphism $\vphi : X \to Y$ (Figure 3). The ramification locus of the latter has infinitely many connected components, each of them being the image of some~$\Gamma_{i}$ by the embedding $A_{i} \hookrightarrow X$. Note that these images do not meet. Also remark that the curve~$X$ is not finite since it contains infinitely many ``loops''.

\begin{center}
  \begin{figure}[h]
  \begin{overpic}[width=0.7\textwidth,tics=10]{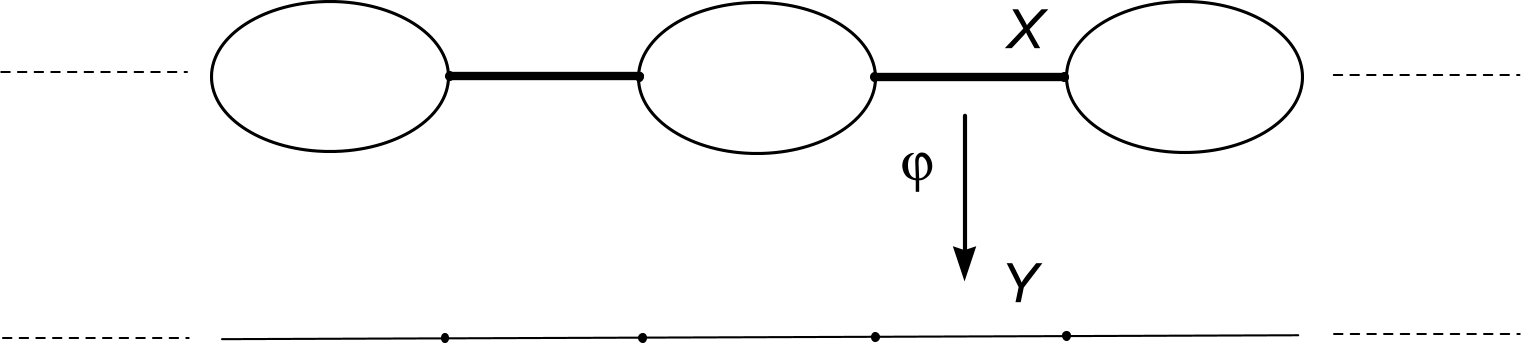}
  \end{overpic}
  \caption{Induced morphism (after gluing) $\vphi:X\to Y$.}
  \end{figure}
 \end{center}

\smallbreak

One may also construct an example of a finite morphism with infinitely many connected components in the ramification locus by considering some cover of the open unit disc $D(0,1^-)$, with coordinate~$T$, of the form $S^2 = F(T)$ for some well chosen $F \in \sO(D(0,1^-))$. We refer to \cite[Exercise~6.1.3.5]{TemkinIntroduction} for some more details.

\subsection{Morphisms of discs and annuli}

\begin{lemma}\label{lemma:morphisms of discs}
 Let $f:D_1\to D_2$ be a finite morphism of open discs. Then, for every open (resp. closed) disc $D'\subset D_1$, ~$f_{|D'}:D'\to f(D')$ is a finite morphism of open (resp. closed) discs. For every open (resp. closed) disc $D''\subset D_2$, ~$f^{-1}(D'')$ is a disjoint union of open (resp. closed) discs.
\end{lemma}
\begin{proof}
 This is a standard fact proved using valuation polygon of the function $f$ expressed in suitable coordinates on $D_1$ and $D_2$, hence we omit the details.
 \end{proof}
 
 \begin{lemma}\label{lem:type4}
 Let $f: D_{1} \to D_{2}$ be a morphism of open or closed discs. Let~$z\in D_{1}$ be a point of type~4. Then, there exists a closed disc~$D$ in~$D_{1}$ containing~$z$ such that the map~$\mu_{f}$ is constant along the interval joining~$z$ to the boundary of~$D$.
 \end{lemma}
 \begin{proof}
 By Corollary \ref{cor:compa partitions} (b), there are open neighborhoods $U_z$ of $z$ in $D_{1}$ and $U_{f(z)}$ of $f(z)$ in $D_{2}$ such that $f_{|U_z}:U_z\to U_{f(z)}$ is a finite morphism and $f(z)$ has only one preimage in $U_z$. 
 
Let~$D$ be closed disc in~$U_{z}$ containing~$z$. Let~$t$ be a point of type~2 or~3 on the interval joining~$z$ to the boundary of~$D$. There exists a unique disc~$D_{t}$ in~$D_{z}$ that contains~$z$ and whose Shilov boundary is reduced to~$t$. By Lemma~\ref{lemma:morphisms of discs}, the induced morphism $f_{|D_{t}} : D_{t} \to f(D_{t})$ is a morphism of closed discs. It follows that $f^{-1}_{|D_{t}}(f_{|D_{t}}(t)) = \{t\}$. Note that, by choice of~$D$, we also have $f^{-1}_{|D_{t}}(f_{|D_{t}}(z)) = \{z\}$. Using the relationship between the multiplicity of a function at a point and its degree stated after Definition \ref{def:multiplicity}, we deduce that 
\[\mu_f(t)=\deg(f_{|D_{t}})=\mu_{f}(z).\]
 \end{proof}

\begin{lemma}\label{lemma:disc ramification}
Let $f: D_{1} \to D_{2}$ be a morphism of open or closed discs. For each $z\in D_{1}$, the map~$\mu_{f}$ is nondecreasing along the interval joining~$z$ to the boundary of~$D_{1}$.

In particular, if $\deg(f)>1$, then $\cR_{f}(D_{1})$ is nonempty and connected.
\end{lemma}
\begin{proof}
 We assume that the discs are closed, the other case reducing to this one by exhausting open discs by closed discs and using Lemma~\ref{lemma:morphisms of discs}.
 
 For the first part of the statement, let $t\in D_1$ be a point which is on the interval joining~$z$ to the boundary of $D_1$ and let $t_1$ be any point on the interval joining $z$ with $t$. Let $D_t$ be the closed subdisc in $D_1$ whose Shilov point is $t$. We note that $f_{|D_t}:D_t\to f(D_t)$ is a finite morphism of closed discs and that $f_{|D_t}^{-1}(f(t))=\{t\}$. Furthermore, $z$ and~$t_1$ belong to~$D_t$.  Using the properties of the multiplicity function, we obtain
 \[
 \mu_{f}(t)=\mu_{f_{|D_t}}(t)=\deg(f_{|D_t})=\sum_{x\in f^{-1}_{|D_t}(f_{|D_t}(t_1))}\mu_{f_{|D_t}}(x)\geq \mu_{f_{|D_t}}(t_1)=\mu_{f}(t_1).
 \]
 
  \smallbreak
  
 Let us prove the last part of the statement. Assume that~$\deg(f)>1$. Denoting by~$\eta_{D_{1}}$ the Shilov boundary of~$D_{1}$, the previous discussion shows that we have
\[\mu_{f}(\eta_{D_{1}}) = \deg(f),\]
which ensures that $\cR_{f}(D_{1})$ is nonempty. 

Let $z,z' \in \cR_{f}(D_{1})$. Denote by $I_{z}$ (resp. $I_{z'}$) the interval joining~$z$ (resp.~$z'$) to~$\eta_{D_{1}}$. Since $\mu_{f}(z)>1$, the first part of the lemma shows that~$I_{z}$ belongs to~$\cR_{f}(D_{1})$, and similarly for~$I_{z'}$. The result now follows from the fact that $I_{z}$ and~$I_{z'}$ meet.
\end{proof}

\begin{corollary}\label{cor:extending disc ramification}
 Let $f:X\to Y$ be a finite morphism of quasi-smooth $k$-analytic curves and let $D$ be an open disc in $X$ attached to a point $x\in X$ and such that $f_{|D}:D\to f(D)$ is a finite morphism of open discs. If $\deg(f_{|D})>1$, then $x$ is a ramified point. In particular, $x$ and $\cR_f(D)$ belong to the same connected component of the ramification locus $\cR_f(X)$.
\end{corollary}
\begin{proof}
 Let us put $y:=f(x)$ and let $U_x$ and $U_y$ be open neighborhoods of $x$ and $y$, respectively, which satisfy the conditions of Corollary \ref{cor:compa partitions} (b). Under this assumption, we have
 \[\mu_{f}(x) = \deg(f_{|U_{x}}).\]
 Since~$U_{x}$ meets~$D$, we have $\deg(f_{|U_{x}}) \ge \deg(f_{|D}) >1$, hence~$x$ is ramified. 
 
 Moreover, Lemma~\ref{lemma:disc ramification} ensures that $\cR_{f}(D)$ is connected and meets~$U_x$. Since~$U_{x}$ may be chosen as small as we want, it follows that~$x$ belongs to the closure of~$\cR_{f}(D)$, hence to the same connected component of the ramification locus.
\end{proof}

\begin{lemma}\label{lem:annuli}
Let $f:A_1\to A_2$ be a finite morphism of open or closed annuli. Then, we have $f(\Gamma_{A_{1}}) = \Gamma_{A_{2}}$, $f^{-1}(\Gamma_{A_{2}}) = \Gamma_{A_{1}}$ and the induced map $f_{|\Gamma_{A_{1}}} : \Gamma_{A_{1}} \to \Gamma_{A_{2}}$ is one to one.

Moreover, for each sub-annulus~$A'_{1}$ of~$A_{1}$ such that $\Gamma_{A'_{1}} \subseteq \Gamma_{A_{1}}$, the image $f(A'_{1})$ is a sub-annulus of~$A_{2}$ with $\Gamma_{f(A'_{1})}) \subseteq \Gamma_{A_{2}}$ and we have $f^{-1}(f(A'_{1})) = A'_{1}$.
\end{lemma}
\begin{proof}
Again, this is standard and proved using valuation polygons arguments. We refer to \cite[Lemma 1.6.1]{BojRH} for a detailed proof in the strictly $k$-analytic case.
\end{proof}

\begin{lemma}\label{lemma:annulus ramification}
Let $f:A_1\to A_2$ be a finite morphism of open or closed annuli. If $\deg(f)>1$, then $\cR_{f}(A_{1})$ is connected and contains the skeleton of~$A_{1}$.
\end{lemma}
\begin{proof}
Assume that $\deg(f)>1$. By Lemma~\ref{lem:annuli}, for each $x \in \Gamma_{A_{1}}$, we have $f^{-1}(f(x)) = \{x\}$, hence $\mu_{f}(x) = \deg(f)>1$. It follows that $\Gamma_{A_{1}} \subseteq \cR_{f}(A_{1})$.

Let $y \in \cR_{f}(A_{1})\setminus \Gamma_{A_{1}}$. Let~$D_{y}$ be the connected component of $A_{1}\setminus \Gamma_{A_{1}}$ containing~$y$. It is an open disc with a unique endpoint~$z$ in~$A_{1}$. Note that~$z$ belongs to~$\Gamma_{A_{1}}$. The image of~$D_{y}$ is the connected component~$D_{f(y)}$ of $A_{2}\setminus \Gamma_{A_{2}}$ containing~$f(y)$. In particular, this image is an open disc. Moreover, $D_{y}$ is a connected component of $f^{-1}(D_{f(y)})$, hence the morphism $f_{|D_{y}} : D_{y} \to D_{f(y)}$ is finite. Since~$y$ is ramified, we have $\deg(f_{|D_{y}})>1$ and Corollary~\ref{cor:extending disc ramification} ensures that $\cR_{f}(D_{y})$ and~$z$ belong to the same connected component of~$\cR_{f}(A_{1})$. In particular, $y$ and~$\Gamma_{A_{1}}$ belong to the same connected component of~$\cR_{f}(A_{1})$. The result follows.
\end{proof}

Arguing as in the proof of Corollary \ref{cor:extending disc ramification}, one gets the following result.

\begin{corollary}\label{cor:extending annuli ramification}
 Let $f:X\to Y$ be a finite morphism of quasi-smooth $k$-analytic curves. Let $A$ be an open annulus in $X$ attached to a point $x\in X$ and such that $f_{|A}:A\to f(A)$ is a finite morphism of open annuli. If $\deg(f_{|A})>1$, then $x$ is ramified. In particular, $\cR_f(A)$ and $x$ belong to the same connected component of the ramification locus $\cR_f(X)$.
\end{corollary}

\subsection{Main results}

%


Note that, if $X$ is finite, then its boundary~$\partial(X)$ is finite. 

 \begin{thm}\label{thm:ramification finite curves}
 Let $f:W\to V$ be a finite morphism of finite $k$-analytic curves with no projective connected components. Then, we have
 \[
 r'_f(W)\leq \#\partial(W)+e^0(W)+g_t(W)-\#\pi_{0}(W)+\sum_{w\in \partial(W)}\delta_m(w),
 \]
 hence
 \[
 r_f(W)\leq \#\partial(W)+e^0(W)+g_t(W)-\#\pi_{0}(W)+\sum_{w\in \partial(W)}\delta_m(w)+\sum_{w\in \cE^0_f(W)}\delta_m(w).
 \]
\end{thm}
\begin{proof}
We may assume that~$W$ and~$V$ are connected. Let $\cS$ and $\cT$ be strictly $f$-compatible triangulations of $W$ and $V$, respectively (see Definition \ref{def:comp triang}), such that $\Gamma_{\cS}$ and $\Gamma_{\cT}$ are $f$-compatible. We argue by induction on $\#\cT$.

 If $\#\cT=0$, then $\#\cS=0$ and this is possible only if $W$ and $V$ are simultaneously open annuli or open discs, and these cases are subject of Lemmas \ref{lemma:disc ramification} and \ref{lemma:annulus ramification} (We cannot have a finite morphism of an annulus to a disc and empty triangulations because we asked that triangulations induce $f$-compatible skeleta). 
   
 \smallbreak

 Suppose that $\#\cT>0$ and that the theorem holds for every finite morphism of finite, non-projective curves $f':W'\to V'$, where $W'$ and $V'$ admit strictly $f'$-compatible triangulations $\cS'$ and $\cT'$, respectively, with $\Gamma_{\cS'}$ and $\Gamma_{\cT'}$ $f$-compatible, such that $\#\cT'<\#\cT$. 
 
 The first step of the proof is to suitably partition $W$ and $V$ according to the triangulations~$\cS$ and $\cT$. For this, let $\cA$ be the family of boundary open annuli in $V\setminus \cT$, that is, $\cA$ is the finite set of all connected components in $V\setminus\cT$ that are open annuli and which are not relatively compact in $V$. Similarly, let $\cA'$ be the set of boundary open annuli in $W\setminus\cS$. (If $W$ and $V$ are compact, then $\cA$ and $\cA'$ are empty.) By Corollary~\ref{cor:compa partitions}, $\cA'$ is exactly the set of connected components of the $f^{-1}(A)$'s, with~$A$ running through~$\cA$. 
  
 Let $\cT'_1$ be the set of endpoints of annuli in $\cA$ and let  $\cT_1\subset\cT$ be the union of $\cT'_1$ and~$\partial(V)$. Similarly, let $\cS'_1$ be the set of endpoints of annuli in $\cA'$ and let $\cS_1\subset\cS$ be the union of $\cS'_1$ and $\partial(W)$. Note that $\cS_1=f^{-1}(\cT_1)$, $\cS'_1=f^{-1}(\cT'_1)$ and $\partial(W)=f^{-1}(\partial(V))$. 
  
 For each $w\in \cS'_1$, let $\cA'_w$ be the set of boundary annuli in $W\setminus\cS$ with endpoint~$w$. Let $\cS''_1\subset \cS'_1$ be the set of those points $w\in\cS'_{1}$ that are the endpoints of exactly one boundary annulus in $W\setminus\cS$.  
  
 Finally, we put $V_0:=V\setminus\bigcup_{A\in\cA}A$ and $W_0:=W\setminus\bigcup_{A'\in\cA'}A'$. The morphism $f$ restricts to a finite morphism $f_0:W_0\to V_0$ of affinoid curves and, by construction, $\cS$ and $\cT$ is a pair of strictly $f_0$-compatible triangulations of $W_0$ and $V_0$, respectively. We also note that $g_t(W_0)=g_t(W)$.
   
   By Lemma \ref{lem:BM for curves}, there are only finitely many connected components $W_1,\dotsc,W_{t}$ of $W_0\setminus \cS_1$ which are not isomorphic to open discs. For each $i=1,\dotsc, t$, $V_{i} := f(W_{i})$ is a connected component of $V_{0}\setminus\cT_{1}$ and $W_{i}$ is a connected component of $f^{-1}(V_{i})$. We denote by $f_{i} : W_{i} \to V_{i}$ the finite morphism induced by~$f$.   
 
Let $\cR_0$ be the set of relatively compact connected components of the ramification locus $\cR_f$ that are disjoint from~$\cS_{1}$. Let $R_{0}\in\cR_{0}$. It is contained in some connected component~$U$ of $W\setminus \cS_{1}$ and it is necessarily relatively compact in it (as the ramification locus is closed and the boundary of~$U$ belongs to~$\cS_{1}$). Therefore, by  Corollaries~\ref{cor:extending disc ramification} and \ref{cor:extending annuli ramification}, $U$ can neither be an open disc nor an open annulus (and in particular, not a boundary annulus). It follows that it is one of the~$W_{i}$'s. It follows that
\[\#\cR_{0} \le \sum_{i=1}^t r'_{f_{i}}(W_{i}).\]
 
Let $\cR_1$ be the set of relatively compact connected component of the ramification locus $\cR_f$ that meet~$\cS_{1}$. Let $R_{1}\in\cR_{1}$. Note that it cannot contain any point $w\in \cS''_{1}$. Indeed, if it were the case, denoting by~$A'_{w}$ the unique boundary annulus with endpoint~$w$, we would have $\deg(f_{|A'_{w}}) = \deg_{w}(f) >1$ and Corollary~\ref{cor:extending annuli ramification} would prevent~$R_{1}$ from being relatively compact. It follows that
\[\#\cR_{1} \le \sum_{w\in \cS_1\setminus \cS''_1}\delta_m(w).\]

Finally, we deduce that 
 \begin{equation*}\label{eq:transition}
 r'_f(W)\leq\sum_{i=1}^t r'_{f_{i}}(W_{i})+\sum_{w\in \cS_1\setminus \cS''_1}\delta_m(w). 
 \end{equation*}

 To continue the proof, we note that, by induction hypothesis, for each morphism $f_{i}:W_i\to V_i$, the result of the theorem holds, as $\cS\cap W^i_j$ and $\cT\cap V_i$ are strictly $f_{i}$-compatible triangulations with corresponding skelta $f$-compatible and $\#(\cT\cap V_i)<\#\cT$. Hence, we have
\[
 r'_{f_{i}}(W_i)\leq e^0(W_i)+g_t(W_i)-1.
\]
Using Lemma~\ref{lem:BM for curves} for~$W_{0}$ and~$\cS_{1}$, we deduce that
  \begin{align*}
 r'_f(W)&\leq\sum_{i=1}^t (e^0(W_i)+g_t(W_i)-1)+\sum_{w\in \cS_1\setminus \cS''_1}\delta_m(w) \nonumber\\
 &\leq g_t(W_0)-1+\#\cS_1+\sum_{w\in \cS_1\setminus \cS''_1}\delta_m(w) \nonumber\\
 &\leq g_t(W_0)-1+\sum_{w\in \cS_1\setminus \cS''_1}(\delta_m(w)+1)+\#\cS''_1  \nonumber\\
 &\leq g_t(W_0)-1+\sum_{w\in \partial(W)}(\delta_m(w)+1)+\sum_{w\in \cS'_1\setminus \cS''_1}(\delta_m(w)+1)+\#\cS''_1. \label{eq:transition2}
 \end{align*}
To finish the proof, it is enough to note that, for each $w\in \cS'_1\setminus\cS''_1$, we have $\delta_m(w)+1\leq 2 \le \#\cA'_w$ and that 
\[e^0(W) = \sum_{w\in \cS'_1}\#\cA'_w = \sum_{w\in \cS'_1\setminus\cS''_1}\#\cA'_w+\#S''_1.\]
\end{proof}

 \begin{remark}
 The result is only interesting when~$f$ is generically \'etale (and similarly for all that follow). Indeed, if~$W$ is connected and~$f$ is inseparable, then we have $\Rc_{f}(W)=W$. 
 \end{remark}

We list below two particular cases of the previous theorem.

\begin{corollary}\label{cor:main}
 Let $f:W\to V$ be a finite morphism of wide open curves. Then, we have
 \[
 r'_f(W)\leq e^0(W)+g_t(W)-\#\pi_0(W),
 \]
 hence
 \[
 r_f(W)\leq e^0(W)+g_t(W)-\#\pi_0(W)+\sum_{v\in\cE^0(W)}\delta_m(A_v).
 \]
\end{corollary}

\begin{corollary}\label{cor:affinoid curves}
 Let $f:W\to V$ be a finite morphism of quasi-smooth $k$-affinoid curves, and let $\sh(W)$ be the Shilov boundary of $W$. Then, we have
 \[
 r_f(W)\leq \#\sh(W)+g_t(W)-\#\pi_{0}(W)+\sum_{w\in\sh(W)}\delta_m(w).
 \]
\end{corollary}

For finite morphisms of smooth, projective $k$-analytic curves, we have the following result.

\begin{thm}\label{cor:generalization}
 Let $f:W\to V$ be a finite morphism of smooth projective $k$-analytic curves. For each connected component~$C$ of~$V$, let $v_{C}\in C$ be a point of type 1, 2 or 3 and let $n(v_{C})$ be the number of components of $\cR_f(W)$ that intersect $f^{-1}(v_{C})$. Then, we have
\[
 r_f(W)\leq g_t(W)-\#\pi_0(W)+\sum_{C \in \pi_{0}(V)} (\#f^{-1}(v_{C})+n(v_{C})).
 \]
 In particular, we have 
 \[
 r_f(W)\leq g_t(W)-\#\pi_0(W)+\#\pi_{0}(V) \deg(f).
 \]
\end{thm}
\begin{proof}
We may assume that~$W$ and~$V$ are connected and will write~$v$ instead of~$v_{V}$. Similarly as before, let $W_1,\dots,W_t$ be all the connected components of $W\setminus f^{-1}(v)$ that are not open discs. We note that, for each~$i$, $W_i$ is a finite curve with no proper connected components and that $f_{|W_i}:W_i\to f(W_i)$ is a finite morphism. By Corollary~\ref{cor:extending disc ramification}, we have
 \[  r_f(W) \leq \sum_{i=1}^tr'_f(W_i)+n(v).\]
Using  Corollary \ref{cor:main} and Lemma \ref{lem:BM for curves}, we find
 \begin{align*}
  r_f(W)  &\leq \sum_{i=1}^t(g_t(W_i)+e^0(W_i)-1)+n(v)\\
  &\leq g_t(W)-\#\pi_0(W)+\#f^{-1}(v)+n(v).
 \end{align*}
 
Let us now prove the second formula. If~$f$ is inseparable, then we have $r_{f}(W)=1$ and the result holds. If~$f$ is generically \'etale, there exists a point $v\in V$ of type~1 that is split. In this case, we have $\# f^{-1}(v)=\deg(f)$ and $n(v)=0$ and the result follows.
 \end{proof}
 
\begin{rmk}
If we specialize our result to the case where $W=V=\P^1_k$, the last formula gives $r_{f}(W)\le \deg(f)-1$, and we recover \cite[Theorem A]{XF}. 
\end{rmk}

\begin{corollary}
 Let $f:W\to V$ be a finite morphism of smooth projective $k$-analytic curves. If~$V$ is connected and if there exists a totally ramified point (\textit{i.e.} a point with only one preimage), then we have 
 \[
 r_f(W)\leq g_t(W)+1.
 \]
\end{corollary}
\begin{proof}
Assume that~$V$ is connected and that there exists a totally ramified point. Note that this forces~$W$ to be connected. By Lemma~\ref{lem:type4}, there exists a totally ramified point $v\in V$ that is of type~1, 2 or~3. The result now follows from   Theorem~\ref{cor:generalization}, using the fact that $\# f^{-1}(v)=n(v)=1$.
\end{proof}

\begin{rmk}
If we specialize our result to the case where $W=V=\P^1_k$ and there exists a totally ramified point, we get $r_{f}(W)\le 1$, and we recover \cite[Theorem C]{XF}.
\end{rmk}

\bibliographystyle{alpha}
\bibliography{biblio}

\begin{thebibliography}{ABBR15}

\bibitem[ABBR15]{ABBR}
Omid {Amini}, Matthew {Baker}, Erwan {Brugall\'e}, and Joseph {Rabinoff}.
\newblock {Lifting harmonic morphisms. I: Metrized complexes and Berkovich
  skeleta.}
\newblock {\em {Research in the Mathematical Sciences}}, 2:67, 2015.

\bibitem[Ber90]{Ber90}
Vladimir~G Berkovich.
\newblock {\em Spectral Theory and Analytic Geometry Over Non-Archimedean
  Fields}.
\newblock American Mathematical Society, 1990.

\bibitem[Ber93]{BerCoh}
Vladimir~G Berkovich.
\newblock {\'E}tale cohomology for non-archimedean analytic spaces.
\newblock {\em Publications Math{\'e}matiques de l'IH{\'E}S}, 78(1):5--161,
  1993.

\bibitem[Boj]{BojRH}
Velibor Bojkovi\'c.
\newblock Riemann-{H}urwitz formula for finite morphisms of $p$-adic curves.
\newblock {\em To appear in Mathematische Zeitschrift}.

\bibitem[BR10]{BR}
Matthew Baker and Robert Rumely.
\newblock {\em Potential theory and dynamics on the {B}erkovich projective
  line}, volume 159 of {\em Mathematical Surveys and Monographs}.
\newblock American Mathematical Society, Providence, RI, 2010.

\bibitem[CTT16]{CTT14}
Adina {Cohen}, Michael {Temkin}, and Dmitri {Trushin}.
\newblock {Morphisms of Berkovich curves and the different function.}
\newblock {\em {Advances in Mathematics}}, 303:800--858, 2016.

\bibitem[Duc]{Duc-book}
Antoine Ducros.
\newblock La structure des courbes analytiques.
\newblock {\em Manuscript available at
  \url{http://www.math.jussieu.fr/~ducros/trirss.pdf}}.

\bibitem[Fab13a]{XF}
Xander Faber.
\newblock Topology and geometry of the {B}erkovich ramification locus for
  rational functions, {I}.
\newblock {\em Manuscripta Mathematica}, 142(3-4):439--474, 2013.

\bibitem[Fab13b]{XF2}
Xander Faber.
\newblock Topology and geometry of the {B}erkovich ramification locus for
  rational functions, {II}.
\newblock {\em Mathematische Annalen}, 356(3):819--844, 2013.

\bibitem[Meh18]{MehmetiPatching}
Vler\"e Mehmeti.
\newblock Patching over berkovich curves and quadratic forms.
\newblock arXiv, 2018.
\newblock \url{https://arxiv.org/abs/1711.00341}.

\bibitem[RV18]{RodriguezNormal}
Rita Rodr\'iguez~V\'azquez.
\newblock Non-archimedean normal families.
\newblock arXiv, 2018.
\newblock \url{https://arxiv.org/abs/1607.05976}.

\bibitem[Tem15]{TemkinIntroduction}
Michael Temkin.
\newblock Introduction to {B}erkovich analytic spaces.
\newblock In {\em Berkovich spaces and applications}, volume 2119 of {\em
  Lecture Notes in Mathematics}, pages 3--66. Springer, Cham, 2015.

\bibitem[Tem17]{temkinuniformization}
Michael Temkin.
\newblock Metric uniformization of morphisms of {B}erkovich curves.
\newblock {\em Advances in Mathematics}, 317:438 -- 472, 2017.

\bibitem[VdP90]{VdP90}
Marius Van~der Put.
\newblock De {R}ham cohomology of affinoid spaces.
\newblock {\em Compositio Mathematica}, 73(2):223--239, 1990.

\end{thebibliography}

\end{document}